\documentclass{elsarticle} 
\usepackage{amsmath,amsthm,amssymb}
\usepackage{enumerate}
\newcommand{\set}[1]{\left\{#1\right\}}
\newcommand{\norm}[1]{\left\lVert#1\right\rVert}
\newcommand{\snorm}[1]{\big\lVert#1\big\rVert}

\newcommand{\ind}[1]{1\hspace{-.28em}\mathrm{I}_{#1}}
\newcommand{\abs}[1]{\left\vert#1\right\vert}
\newcommand{\aabs}[1]{\big\lvert#1\big\rvert}

\newcommand{\eps}{\varepsilon}
\newcommand{\ex}[1]{\mathsf{E}\left[#1\right]}
\newcommand{\oX}{\overline X}
\newcommand{\oZ}{\overline Z}

\newtheorem{theorem}{Theorem}[section]
\newtheorem{lemma}{Lemma}[section]
\newtheorem{proposition}{Proposition}[section]
\newtheorem{corollary}{Corollary}[section]
\theoremstyle{remark}
\newtheorem{remark}{Remark}[section]
\theoremstyle{definition}

\newcommand{\R}{\mathbb{R}}

\newcommand{\cI}{\mathcal{I}}

\begin{document}
\title{Mixed stochastic differential  equations with long-range dependence: existence, uniqueness and convergence of solutions}
\author{Yuliya Mishura}
\ead{myus@univ.kiev.ua}
\author{Georgiy Shevchenko\corref{cor1}}
\cortext[cor1]{Corresponding author}
\ead{zhora@univ.kiev.ua}
\address{Kyiv National Taras Shevchenko University, Department of Mechanics and Mathematics, Volodymyrska 64, 01601 Kyiv, Ukraine}

\begin{abstract}
For a mixed stochastic differential equation
involving standard Brownian motion and an almost
surely H\"older continuous process $Z$ with H\"older exponent $\gamma>1/2$, we establish a new result on its unique solvability.
We also establish an estimate for difference of solutions to such equations with different processes $Z$ and deduce
a corresponding limit theorem. As a by-product, we obtain a result on existence of moments of a solution to a mixed equation
under an assumption that $Z$ has certain exponential moments.
\end{abstract}
\maketitle

\section*{Introduction}

In this paper we study the following mixed stochastic differential equation:
 \begin{equation}\label{main} X_t =X_0 +\int_0^t a(s,X_{s})ds+
\int_0^tb(s,X_{s})dW_s+\int_0^tc(s,X_{s})dZ_s,\
t\in[0,T],\end{equation}
where  $W$ is a standard Wiener process, and $Z$ is an almost surely H\"older continuous process with H\"older exponent $\gamma>1/2$. The processes
$W$ and $Z$ can be dependent.

The motivation to consider such equations comes, in particular,  from financial mathematics. When it is necessary to model
randomness on a financial market, it is useful to distinguish between two main sources of this randomness.
The first source is the stock exchange itself with thousands of agents. The noise coming from this source can
be assumed white and is best modeled by a Wiener process.
The second source has the financial and economical background. The random noise coming from this source
usually has a long range dependence property, which can be  modeled by a fractional Brownian
motion $B^H$ with the Hurst parameter $H>1/2$ or by a multifractional Brownian motion with the Hurst function
uniformly greater than $1/2$. Most of  long-range-dependent processes have one thing in common: they are
H\"older continuous with exponent greater than $1/2$, and this is the reason to consider a rather general
equation \eqref{main}.

Equation \eqref{main} with $Z=B^H$, a  fractional Brownian motion, was first considered  in \cite{kubilius}, where existence and uniqueness
of solution was proved for time-independent coefficients and zero drift. For inhomogeneous coefficients, unique solvability
was established in \cite{mis08} for  $H\in (3/4,1)$ and bounded coefficients,  in \cite{guernual} for any $H>1/2$, but under the assumption that $W$ and $B^H$ are independent, and in \cite{missh11} for any $H>1/2$, but bounded coefficient $b$. In this paper we generalize the last
result replacing the boundedness assumption by the linear growth.

There is, however, an obstacle to use  equation \eqref{main} in applications because it is very hard to analyze with
standard tools of stochastic analysis. The main reason for this is that the two stochastic integrals in \eqref{main}
have very different nature. The integral with respect to the Wiener process is It\^o integral,
and it is best analyzed in a mean square sense, while the integral with respect
to $Z$ is understood in a pathwise sense, and all estimates are pathwise with random constants.
So there is a need for good approximations for such equations. One way to approximate is to replace
integrals by finite sums, this leads to Euler approximations. For equation \eqref{main} such approximations were
considered in \cite{mbfbm-euler}, where a sharp estimate for the rate of convergence was obtained.
Another way is to replace process $Z$ by a smooth process $\oZ$, transforming equation \eqref{main}
into a usual It\^o stochastic differential equation with random drift $a(s,x) + c(s,x) Z'_s$. Since there is a
well-developed theory for It\^o equations, such smooth approximations may prove very useful in applications.

The paper is organized as follows. In Section 1, we give basic facts about
integration with respect to fractional Brownian motion and formulate main hypotheses. In Section 2, we
establish auxiliary results. As a by-product, we obtain a result on existence of moments of a solution to a mixed equation
under an assumption that $Z$ has certain exponential moments, which is satisfied, for example, by a fractional Brownian motion
with Hurst parameter $H>3/4$. Section 3 contains the result about existence and uniqueness of solution to equation \eqref{main}.
In Section 4, we estimate a difference between two solutions of equations \eqref{main} with different processes $Z$
and deduce  a limit theorem for equation \eqref{main} from this estimate.

\section{Preliminaries}

Let ($\Omega,\mathcal{F},\set{\mathcal{F}_t}_{t\in[0,T]},P$) be a complete
probability space equipped with a filtration satisfying standard
assumptions, and  $\{W_t,
t \in [0,T]\}$   be a standard $\mathcal F_t$-Wiener process.
Let also $\{Z_t, t \in [0,T]\}$ be an $\mathcal F_t$-adapted stochastic process, which is
almost surely H\"older continuous with exponent $\gamma>1/2$.
We consider a mixed stochastic differential equation \eqref{main}.
The integral w.r.t.\ Wiener process $W$ is the standard It\^o integral, and the integral w.r.t.\ $Z$ is pathwise generalized
Lebesgue--Stieltjes integral (see \cite{NR,Zah98a}), which is defined as follows.
Consider two  continuous functions  $f$ and $g$, defined on some interval $[a,b]\subset \mathbb{R}$.
For $\alpha\in (0,1)$ define fractional derivatives
\begin{gather*}
\big(D_{a+}^{\alpha}f\big)(x)=\frac{1}{\Gamma(1-\alpha)}\bigg(\frac{f(x)}{(x-a)^\alpha}+\alpha
\int_{a}^x\frac{f(x)-f(u)}{(x-u)^{1+\alpha}}du\bigg)1_{(a,b)}(x),\\
\big(D_{b-}^{1-\alpha}g\big)(x)=\frac{e^{-i\pi
\alpha}}{\Gamma(\alpha)}\bigg(\frac{g(x)}{(b-x)^{1-\alpha}}+(1-\alpha)
\int_{x}^b\frac{g(x)-g(u)}{(u-x)^{2-\alpha}}du\bigg)1_{(a,b)}(x).
\end{gather*}

Assume that
 $D_{a+}^{\alpha}f\in L_1[a,b], \ D_{b-}^{1-\alpha}g_{b-}\in
L_\infty[a,b]$, where $g_{b-}(x) = g(x) - g(b)$.
Under these
assumptions, the generalized (fractional) Lebesgue-Stieltjes integral $\int_a^bf(x)dg(x)$ is defined as
\begin{equation}\label{general l-s}\int_a^bf(x)dg(x)=e^{i\pi\alpha}\int_a^b\big(D_{a+}^{\alpha}f\big)(x)\big(D_{b-}^{1-\alpha}g_{b-}\big)(x)dx.
\end{equation}

In view of this, we will consider the following  norms for $\alpha \in(1-H,1/2)$:
\begin{gather*}
\norm{f}^2_{2,\alpha;t} = \int_0^t \norm{f}_{\alpha;s}^2 g(t,s)ds,\\
\norm{f}_{\infty,\alpha;t} = \sup_{s\in [0,t]} \norm{f}_{\alpha;s},
\end{gather*}
where $g(t,s) = s^{-\alpha} + (t-s)^{-\alpha-1/2}$ and
 $$\norm{f}_{\alpha;t} = \abs{f(t)} + \int_0^t \abs{f(t)-f(s)}(t-s)^{-1-\alpha} ds.$$
Also define a seminorm
\begin{gather*}
\norm{f}_{0,\alpha;t} = \sup_{0\le u<v<t} \left(\frac{\abs{f(v)-f(u)}}{(v-u)^{1-\alpha}} + \int_u^v \frac{\abs{f(u)-f(z)}}{(z-u)^{2-\alpha}}dz\right).
\end{gather*}

Recall that by our assumption $Z$ is almost surely $\gamma$-H\"older continuous with $\gamma>\frac12$. Hence it is easy to see that for any $\alpha\in (1-\gamma,1/2)$
$$
\sup_{0\leq
u< v\leq t}\aabs{\big(D_{v-}^{1-\alpha}Z_{t-}\big)(u)}\le \norm{Z}_{0,t}
<\infty.
$$
Thus, we can define the integral with respect to $Z$ by \eqref{general l-s}, and it  admits  the following estimate for $0\le a<b\le t$:
\begin{equation}\label{ineq}
\abs{\int_a^b f(s)dZ_s} \le
C_{\alpha}\norm{Z}_{0,t} \int_a^b\bigg(\abs{f(s)}
(s-a)^{-\alpha}+\int_a^s \abs{f(s)-f(z)}
(s-z)^{-\alpha-1}dz\bigg)ds.
\end{equation}
for any
$\alpha\in(1-\gamma,1/2) $, $t>0$, $u\le v\le t$ and any $f$ such that the right-hand side of this inequality is finite.

We will assume that for some $K>0$ and for any 
$t,s\in[0,T]$, $x,y\in\mathbb R$, $\beta>1/2$
\begin{equation}\label{assumpt}
\begin{aligned}
&| a(t,x)|  +|b(t,x)| + |c(t,x)| \leq K(1+\abs{x}),\quad
\abs{\partial_x c(t,x)} \leq K,\\
&|a(t,x)-a(t,y)|+|b(t,x)-b(t,y)|+|\partial_x{}c(t,x)-\partial_x{}c(t,y)|\le K|x-y|,\\
&|a(s,x)-a(t,x)|+|b(s,x)-b(t,x)|+|c(s,x)-c(t,x)|+|\partial_x{}c(s,x)-\partial_x{}c(t,x)|\le{K}|s-t|^{\beta}.
\end{aligned}
\end{equation}
It was proved in \cite{missh11} that equation \eqref{main} is uniquely solvable when these assumptions
hold and if additionally $b$ is bounded: for some $K_1>0$
\begin{equation}
\label{boundb}
|b(t,x)|\le K_1.
\end{equation}
The reason for us to formulate this assumption individually is that we are going to drop this assumption.

\section{Auxiliary results}
\begin{lemma}
Let $g:[0,T]\to \R$ be a $\gamma$-H\"older continuous function. Define for $\eps>0$
$g^\eps(t) = \eps^{-1} \int_{0\vee t-\eps}^t g(s) ds$. Then for $\alpha\in(1-\gamma,1)$
$$
\norm{g-g^\eps}_{0,\alpha;T}\le C K_\gamma(g) \eps^{\gamma+\alpha-1},
$$
where $K_\gamma(g) = \sup_{0\le s<t\le T} \abs{g(t)-g(s)}/(t-s)^{\gamma}$ is the H\"older constant of $g$.
\end{lemma}
\begin{proof}Without loss of generality assume that $g(0)=0$.
To simplify the notation, assume that $g(x)=0$ for $x< 0$.
Take any $t,s\in[0,T]$. For $\abs{t-s}\ge \eps$
\begin{gather*}
\abs{g(t)-g^\eps(t)-g(s)+g^\eps(s)} = \eps^{-1}\abs{\int_{t-\eps}^t \big(g(t)-g(u)\big)du - \int_{s-\eps}^s \big(g(s)-g(v)\big)dv}\\
\le K_\gamma(g)\eps^{-1}  \abs{\int_{t-\eps}^t (t-u)^\gamma du} +\abs{\int_{s-\eps}^s (s-v)^\gamma dv}
\le C K_\gamma(g) \eps^\gamma;
\end{gather*}
for $\abs{t-s}<\eps$
\begin{gather*}
\abs{g(t)-g^\eps(t)-g(s)+g^\eps(s)} \le \abs{g(t)-g(s)} +
 \eps^{-1}\abs{\int_{-\eps}^0 \big(g(t+u)-g(s+u)\big)du}
\le C K_\gamma(g)\abs{t-s}^{\gamma},
\end{gather*}
consequently
\begin{equation}
\abs{g(t)-g^\eps(t)-g(s)+g^\eps(s)}\le C K_\gamma(g)\big(\eps\wedge\abs{t-s}\big)^{\gamma}.
\end{equation}
Now write
\begin{gather*}
\norm{g-g^{\eps}}_{0,\alpha;T} \le A^\eps+B^\eps,
\end{gather*}
where
\begin{gather*}
A^\eps = \sup_{0\le u<v\le T} \frac{\abs{g(u)-g^\eps(u)-g(v)+g^\eps(v)}}{(v-u)^{1-\alpha}} \\
\le C K_\gamma(g) \sup_{0\le u<v\le T} (v-u)^{\alpha-1}\big(\eps\wedge\abs{v-u}\big)^{\gamma}
\le C K_\gamma(g)\eps^{\gamma+\alpha-1},\\
B^\eps  = \sup_{0\le u<v\le T}\abs{\int_u^v \frac{g(u)-g^\eps(u)-g(x)+g^\eps(x)}{(x-u)^{2-\alpha}}dx}\\
\le C K_\gamma(g)\sup_{0\le u<v\le T}\int_u^v \frac{\big((x-u)\wedge \eps\big)^{\gamma}}{(x-u)^{2-\alpha}}dx
\\
 \le C K_\gamma(g)\sup_{0\le u<v\le T} \big((v-u)\wedge \eps\big)^{\gamma+\alpha-1}\le C K_\gamma(g)\eps^{\gamma+\alpha-1}.
\end{gather*}
\end{proof}

\begin{corollary}
Let $g:[0,T]\to \R$ be a $\gamma$-H\"older continuous function with $g(0)=0$. There exists a sequence of continuously
differentiable functions $\{g^n,n\ge 1\}$ such that for any $\alpha\in(1-\gamma,1)$ \ $\norm{g-g_n}_{0,\alpha;T}\to 0$,
$n\to\infty$. One possible choice of such sequence is $g_n(t) = a_n^{-1} \int_{0\vee t-a_n}^t g(s) ds$, where $a_n\downarrow 0$, $n\to\infty$.
\end{corollary}

Further throughout the paper there will be no ambiguity
about $\alpha$, so for the sake of shortness we will usually abbreviate $\norm{f}_t = \norm{f}_{\alpha;t}$ and $\norm{f}_{x,t}=\norm{f}_{x,\alpha;t}$, where $x\in\set{0,2,\infty}$.

\begin{lemma}\label{lemma-preaprior} Under assumptions \eqref{assumpt} and \eqref{boundb}
\begin{gather*}
\norm{X}_t\le C \norm{Z}_{0,t} \bigg(1 + \int_0^t \norm{X}_{s}  \big(s^{-\alpha} + (t-s)^{-2\alpha}\big) ds \bigg) + I_b(t),
\end{gather*}
where
$I_b(t) = \norm{\int_0^{\cdot} b(s,X_s) dW_s}_{t}$.
\end{lemma}
\begin{proof}
Write
$\norm{X}_t \le \abs{X_0} + I_a(t) + I_b(t) + I_c(t)$, where $
I_a(t) = \norm{\int_0^{\cdot} a(s,X_s) ds}_{t}$, $I_c(t)=\norm{\int_0^{\cdot} c(s,X_s) dZ_s}_{t}$. Denote for shortness
$\Lambda = \norm{Z}_{0,t}$.

Estimate
\begin{gather*}
I_a(t) \le C\bigg(\int_0^t \abs{a(s,X_s)} ds + \int_0^t \int_s^t\abs{a(u,X_u)}du(t-s)^{-1-\alpha} ds \bigg)\\
 \le C\bigg(\int_0^t \big(1+ \abs{X_s}\big)  ds +\int_0^t \int_s^t\big(1+ \abs{X_u}\big) du(t-s)^{-1-\alpha} ds \bigg)\\
 \le C\bigg(1+ \int_0^t \abs{X_s} ds + \int_0^t \abs{X_u} (t-u)^{-\alpha} du\bigg)\le C\bigg(1+ \int_0^t \norm{X}_{s}(t-s)^{-\alpha} ds\bigg).
\end{gather*}
Further,
\begin{gather*}
I_c(t) \le I_c'(t) + I_c''(t),
\end{gather*}
where
\begin{gather*}
I_c'(t) = \abs{\int_0^t c(s,X_s) dZ_s}\le C\Lambda\int_0^t\bigg(\big(1+ \abs{X_s}\big)s^{-\alpha} + \int_0^s \abs{X_s-X_u}(s-u)^{-1-\alpha}du\bigg)ds \\
\le C\Lambda\bigg(1+ \int_0^t \norm{X}_{s} s^{-\alpha} ds\bigg),\\
I_c''(t) = \int_0^t \abs{\int_s^t c(u,X_u) dZ_v}(t-s)^{-1-\alpha}ds \\
\le C\Lambda\int_0^t  \int_s^t \bigg(1+ \abs{X_v}(v-s)^{-\alpha} + \int_s^v \abs{X_v-X_z}(v-z)^{-1-\alpha} dz\bigg)dv (t-s)^{-1-\alpha}ds \\
\le C\Lambda \bigg(1 + \int_0^t \int_s^t \norm{X}_{v} (v-s)^{-\alpha} dv (t-s)^{-1-\alpha} ds \bigg)\\
 =   C\Lambda \bigg(1 + \int_0^t \norm{X}_{v} \int_0^v  (v-s)^{-\alpha} (t-s)^{-1-\alpha}ds\, dv \bigg)\le
  C\Lambda \bigg(1 + \int_0^t \norm{X}_{v}  (t-v)^{-2\alpha} dv \bigg) .
\end{gather*}
Combining these estimates, we get
\begin{gather*}
\norm{X}_t\le C\Lambda \bigg(1 + \int_0^t \norm{X}_{s}  \big(s^{-\alpha} + (t-s)^{-2\alpha}\big) ds \bigg) + I_b(t).
\end{gather*}
\end{proof}

\begin{proposition}\label{prop-moments}
Under assumptions \eqref{assumpt}, \eqref{boundb} and
\begin{equation}\label{expintegr}
\ex{\exp\set{a \norm{Z}_{0,T}^{1/(1-2\alpha)}}}<\infty,
\end{equation}
all moments of $X$ are bounded,
 moreover, $\ex{\norm{X}_{\infty,T}^p}<\infty$ for all $p>0$.
\end{proposition}
\begin{proof}
By the generalized Gronwall lemma from \cite{NR} it follows from Lemma~\ref{lemma-preaprior} that
$$\norm{X}_{t}\le C \norm{Z}_{0,t} \sup_{s\in[0,t]} I_b(s) \exp\set{C \norm{Z}_{0,t}^{1/(1-2\alpha)}},$$
whence
$$\norm{X}_{\infty,T}\le C \norm{Z}_{0,T} \sup_{s\in[0,T]} I_b(s) \exp\set{C \norm{Z}_{0,T}^{1/(1-2\alpha)}},$$
whence the assertion follows, as all moments of $\sup_{s\in[0,T]} I_b(s)$ are bounded due to the Burkholder inequality
and the boundedness of $b$.
\end{proof}
\begin{remark}
The assumption \eqref{expintegr} might seem very restrictive. However, it is true if $Z$ is Gaussian and $\alpha<1/4$ (clearly, such choice of $\alpha$ is possible iff $\gamma>3/4$).
Indeed, it is well known that if  supremum of a Gaussian family  is finite almost surely,
than its square has small exponential moments finite, so we have \eqref{expintegr} since $1/(1-2\alpha)<2$. Examples
of such processes are fractional Brownian motion with Hurst parameter $H>3/4$ and multifractional Brownian motion with
Hurst function whose minimal value exceeds $3/4$.
\end{remark}

For $N\ge 1$ define a stopping time $\tau_N = \inf \set{t: \norm{Z}_{0,t} \ge N}\wedge T$ and a stopped
process $Z^N_t = Z_{t\wedge \tau_N}$, denote
by $X^{N}$ the solution of \eqref{main} with $Z$ replaced by $Z^N$.

\begin{lemma}\label{lem-aprior}
Under assumptions \eqref{assumpt} and \eqref{boundb} it holds
$$
\ex{\norm{X^N}^p_{\infty,T}}<C_{p,N}
$$
with the constant $C_{p,N}$ independent of $Z$ and $K_1$.
\end{lemma}
\begin{proof}
First note that the finiteness of $\ex{\norm{X^N}^p_{\infty,T}}$ can be deduced from Lemma~\ref{lemma-preaprior} exactly the same way as Proposition~\ref{prop-moments}.

Second, it follows from  Lemma \ref{lemma-preaprior}
and  the generalized Gronwall lemma \cite{NR} that
\begin{equation*}
\norm{X^N}_t\le CN \sup_{s\in[0,t]} I^N_b(s) \exp\set{CtN^{1/(1-2\alpha)}}\le C_{N}  \sup_{s\in[0,t]} I^N_b(s),
\end{equation*}
which implies
$$
\norm{X^N}_{\infty,t}^p\le  C_{N,p}  \sup_{s\in[0,t]} \left(I^N_b(s)\right)^p.
$$
%

Write $$\ex{\sup_{s\in[0,T]} \left(I_b^N(t)\right)^p}\le I_b' + I_b'',$$ where, denoting $b^N_u = b(u,X^N_u)$,
\begin{gather*}
I_b' =\ex{\sup_{s\in[0,t]}\abs{\int_0^{s} b^N_u dW_u}^p }
\le C_p\ex{\left(\int_0^t \abs{b_s^N}^2ds\right)^{p/2}} \le C_p\ex{ \left(\int_0^t(1+\norm{X^N}^2_{s})ds\right)^{p/2}}\\
\le C_{p} \left(1+ \ex{\left(\int_0^t \norm{X^N}^2_{s} ds\right)^{p/2}}\right)\le C_{p} \left(1+ \int_0^t \ex{\norm{X^N}^p_{\infty,s}} ds \right) ,
\\
I_b'' = \ex{\sup_{s\in[0,t]}\left(\int_0^s \abs{\int_u^s b^N_z dW_z} (s-u)^{-1-\alpha}du\right)^p }.
\end{gather*}

Obviously, we can assume without loss of generality that $p>4/(1-2\alpha)$. It follows from the Garsia--Rodemich--Rumsey inequality that
for arbitrary
$\eta\in(0,1/2-\alpha)$, $u,s\in[0,t]$
$$\abs{\int_u^s b^N_z dW_z}\le C \xi^N_\eta(t)\abs{s-u}^{1/2-\eta},
\quad \xi^N_\eta(t) = \left(\int_0^t \int_0^t \frac{\abs{\int_x^y b^N_v dW_v}^{2/\eta}} {\abs{x-y}^{1/\eta}}dx\,dy\right)^{\eta/2}.
$$
Setting $\eta=2/p$, we get
$$
\xi^N_\eta(t) = \left(\int_0^t \int_0^t \frac{\abs{\int_x^y b^N_v dW_v}^{p}} {\abs{x-y}^{p/2}}dx\,dy\right)^{1/p}.
$$
Then, similarly to estimate for $I_b'$, we get
\begin{gather*}
\ex{\left(\xi^N_\eta(t)\right)^p}\le C_p\int_0^t \int_0^t \frac{\ex{\abs{\int_x^y b^N_v dW_v}^{p}}} {\abs{x-y}^{p/2}}dx\,dy\\
\le C_p \int_0^t \int_0^y \frac{\ex{\left(\int_x^y (1+\norm{X^N}_{\infty,v}^2)dv\right)^{p/2}}} {(y-x)^{p/2}} dx\,dy\\
\le C_p \int_0^t \int_0^y \frac{(y-x)^{p/2}+\ex{\left(\int_x^y \norm{X^N}_{\infty,v}^2dv\right)^{p/2}}} {(y-x)^{p/2}} dx\,dy \\
\le C_p \int_0^t \int_0^y \frac{(y-x)^{p/2}+(y-x)^{p/2-1}\int_x^y \ex{\norm{X^N}_{\infty,v}^p}dv} {(y-x)^{p/2}} dx\,dy\\
\le C_{p}\left(1+\int_0^t\int_0^y \int_x^y \ex{\norm{X^N}_{\infty,v}^p}dv (y-x)^{-1}dv\,dx\,dy  \right)\\
= C_p \left(1+ \int_0^t \int_0^y \ex{\norm{X^N}_{\infty,v}^p} \log \frac{y}{y-v} dv\,dy \right)\\
\le C_p \left(1+ \int_0^t \ex{\norm{X^N}_{\infty,v}^p}dv\right).
\end{gather*}
whence
$$
I_b'' \le C \ex{\xi^N_\eta(t)^p}\sup_{s\in[0,t]}\left(\int_0^t (t-s)^{-1/2-\eta-\alpha}ds\right)^p \le C_{p} \left(1+ \int_0^t \ex{\norm{X^N}_{\infty,v}^p}dv\right).
$$
Thus, we have the estimate
$$
\norm{X^N}_{\infty,t}^p\le  C_{N,p}  \left(1+ \int_0^t \ex{\norm{X^N}_{\infty,v}^p}dv\right),
$$
from which we derive the desired statement with the help of the Gronwall lemma.
\end{proof}

\section{Existence of solution}

Now we prove existence and uniqueness of solution to equation \eqref{main} without assumption \eqref{boundb}. As above,
we define a stopped
process $Z^N_t = Z_{t\wedge \tau_N}$, where  $\tau_N = \inf \set{t: \norm{Z}_{0,t} \ge N}\wedge T$.
Denote by $X^{N}$ the solution of \eqref{main} with $Z$ replaced by $Z^N$.
\begin{theorem}\label{thm-exists-unique}
If the coefficients of equation  \eqref{main} satisfy conditions \eqref{assumpt}, then it has a unique solution $X$ such that $\norm{X}_{\infty,T}<\infty$ a.s.
\end{theorem}
\begin{proof}
For integer $N\ge 1$, $R\ge 1$  denote $X^{N,R}$ the solution of equation \eqref{main} with process $Z$ replaced by $Z^N$ and
coefficient $b$ replaced by $b\wedge (K(R+1))\vee (-K(R+1))$ (we will call it an $(N,R)$-equation).
Let also $\tau_{N,R} = \inf \set{t: \abs{X^{N,R}_t}
\ge R}\wedge T$. We argue that $X^{N,R'}_t = X^{N,R''}_t$ a.s. for $t< \tau_{N,R'}\wedge \tau_{N,R''}$.

For brevity define $Y_{t,s} = Y_t-Y_s$ and denote $h(t,s)=(t-s)^{-1-\alpha}$, $\ind{t}=\ind{t< \tau_{N,R'}\wedge \tau_{N,R''}}$.  All the constants in this step may depend on $N$ and $R'$, $R''$.

Write
\begin{equation}\label{riznytsia1}
\begin{gathered}  (X^{N,R'}_{t}-X^{N,R''}_{t})\ind{t}
 =\left(\int_0^{t} a_\Delta(s)ds+
\int_0^{t} b_\Delta(s) dW_s+\int_0^{t} c_\Delta(s)dZ^N_s\right)\ind{t} \\=:\left(\cI_a(t)+\cI_b(t)+\cI_c(t)\right)\ind{t},\end{gathered}
\end{equation}
where $d_\Delta(s):=d(s,X^{N,R'})-d(s,X^{N,R''}),\; d\in\{a,b,c\}.$
Due to our hypotheses,  $\abs{d_\Delta(s)} \le C \aabs{X^{N,R'}_{s}-X^{N,R''}_{s}}$.

Define $\Delta_t = \int_0^t \snorm{X^{N,R'}-X^{N,R''}}^2_{s}\ind{s} g(t,s) ds$.
Write
\begin{gather*}
\Delta_t\le 6(I_a' + I_a''+I_b' + I_b''+I_c' + I_c'' ),
\end{gather*}
where $I_d' = \int_0^{t}\cI_d(s)^2 \ind{s}g(t,s) ds$, $I_d''= \int_0^t \big(\int_0^s|\cI_d(s)-\cI_d(u)| h(s,u) du\big)^2\ind{s} g(t,s) ds$, $d\in\set{a,b,c}$.

By the Cauchy--Schwarz inequality, we can write
\begin{equation}\label{Ia}
\cI_a(s)^2\ind{s} \le C\int_0^s \aabs{X^{N,R'}_{u}-X^{N,R''}_{u}}^2\ind{u} du\le C\int_0^s \norm{X^{N,R'}-X^{N,R''}}^2_u\ind{u} du,
\end{equation}
therefore
\begin{gather*}
I_a'\le C  \int_0^t \Delta_s\,g(t,s) ds.
\end{gather*}
Similarly,
\begin{gather*}
I_a''\le  C \int_0^t \left(\int_0^{s} \int_{u}^{s} \aabs{X^{N,R'}_{v}-X^{N,R''}_{v}} dv\, h(s,u)du\right)^2\ind{s} g(t,s)ds \\
\le C \int_0^t\bigg( \int_0^s \aabs{X^{N,R'}_{v}-X^{N,R''}_{v}}\ind{v} (s-v)^{-\alpha} dv \bigg)^2 g(t,s)ds \\
\le C \int_0^t \int_0^s \aabs{X^{N,R'}_{v}-X^{N,R''}_{v}}^2\ind{v} (s-v)^{-\alpha} dv \,g(t,s) ds
\le C\int_0^t\Delta_s\,g(t,s) ds.
\end{gather*}
Further, by \eqref{ineq}, for $s\le t$
\begin{equation*}
\begin{gathered}
\cI_c(s)^2\ind{s}\le C N \bigg[\int_0^{s}\left( |c_\Delta(u)|u^{-\alpha}  +
\int_0^u |c_\Delta(u)-c_\Delta(z)|h(u,z) dz\right)du\bigg]^2\ind{s} \\
\le C  \left[\bigg(\int_0^s |c_\Delta(u)|u^{-\alpha}  du\bigg)^2 +
\bigg(\int_0^s \int_0^u |c_\Delta(u)-c_\Delta(z)|h(u,z) dz\, du\bigg)^2 \right]\ind{s}
=: C (J_c'+J_c'').
\end{gathered}
\end{equation*}
Analogously to $\cI_a$,
\begin{gather*}
J_c'\le C\int_0^s \norm{X^{N,R'}-X^{N,R''}}^2_u\ind{u} u^{-\alpha} du.
\end{gather*}
By Lemma 7.1 of Nualart and R\u a\c scanu  (2002),
the hypotheses (A)--(D) 
imply that
for any $t_1,t_2,x_1,\dots,x_4$
\begin{equation}\label{cocenka}
\begin{gathered}
\abs{c(t_1,x_1)-c(t_2,x_2)-c(t_1,x_3)+c(t_2,x_4)}\le K\abs{x_1-x_2-x_3+x_4}\\
+ K\abs{x_1-x_3}\abs{t_2-t_1}^\beta+ K\abs{x_1-x_3}( \abs{x_1-x_2} + \abs{x_3-x_4}).
\end{gathered}
\end{equation}
Therefore,
\begin{equation*}
\begin{gathered}
|c_\Delta(u)-c_\Delta(z)|= \aabs{c(u,X^{N,R'}_{u})-c(z,X^{N,R'}_{z})-
c(u,X^{N,R''}_{u})+c(z,X^{N,R''}_{z})}\\
\le C\bigg(\aabs{X^{N,R'}_{u,z}-X^{N,R''}_{u,z}}
+ \aabs{X^{N,R'}_{u}-X^{N,R''}_{u}} (u-z)^\beta + \aabs{X^{N,R'}_{u}-X^{N,R''}_{u}}
\Big(\aabs{X^{N,R'}_{u,z}} + \aabs{X^{N,R''}_{u,z}}\Big)\bigg).
\end{gathered}
\end{equation*}
Thus, we have
\begin{gather*}
J_c''\le C \bigg[\int_0^s\int_0^{u}\bigg(\aabs{X^{N,R'}_{u,z}-X^{N,R''}_{u,z}}
+\aabs{X^{N,R'}_{u}-X^{N,R''}_{u}} (u-z)^\beta\\
+ \aabs{X^{N,R'}_{u}-X^{N,R''}_{u}}
\Big(\aabs{X^{N,R'}_{u,z}} + \aabs{X^{N,R''}_{u,z}}\Big)\bigg)h(u,z)dz\ind{u}du\bigg]^2 \le C( H_1 + H_2),
\end{gather*}
where
\begin{gather*}
H_1 =  \bigg(\int_0^s\int_0^{u}\Big(\aabs{X^{N,R'}_{u,z}-X^{N,R''}_{u,z}}
+\aabs{X^{N,R'}_{u}-X^{N,R''}_{u}}(u-z)^\beta\Big)h(u,z)dz\ind{u} du\bigg)^2 \\
 \le C \int_0^s\bigg(\snorm{X^{N,R'}-X^{N,R''}}_u\ind{u} + \aabs{X^{N,R'}_u-X^{N,R''}_u}\ind{u} u^{\beta-\alpha}\bigg)^2 du\le C \int_0^s \snorm{X^{N,R'}-X^{N,R''}}^2_u\ind{u} du,\\
H_2= \left( \int_0^s \aabs{X^{N,R'}_u-X^{N,R''}_u} \int_0^u\Big(\aabs{X^{N,R'}_{u,z}} + \aabs{X^{N,R''}_{u,z}}\Big)h(u,z)dz\ind{u} du\right)^2   \\
\le C \bigg(\int_0^s \aabs{X^{N,R'}_u-X^{N,R''}_u}\big(\snorm{X^{N,R'}}_{\infty,u}+
\snorm{X^{N,R''}}_{\infty,u}\big)\ind{u} du\bigg)^2\\
\le C(R'+R'')^2 \int_0^s \aabs{X^{N,R'}_u-X^{N,R''}_u}^2 \ind{u} du
\le C \int_0^s \snorm{X^{N,R'}-X^{N,R''}}^2_u \ind{u} du.
\end{gather*}
It follows that
\begin{equation}
\label{Ic}
\cI_c(s)^2 \le C \int_0^s \norm{X^{N,R'}-X^{N,R''}}^2_u \ind{u} u^{-\alpha} du.
\end{equation}
Consequently,
$$
I_c'\le C \int_0^t \Delta_s\, g(t,s)ds.
$$
Now by \eqref{ineq} and \eqref{cocenka}
\begin{gather*}
I_c'' \le N \int_0^{t} \bigg(\int_0^{s} \int_{u}^{s} \bigg(\abs{c_\Delta(v)}(v-u)^{-\alpha} + \int_{u}^v \abs{c_\Delta(v)-c_\Delta(z)}h(v,z)dz\bigg)dv \,h(s,u)du\bigg)^2 g(t,s) \ind{s}ds
\\\le C  \int_0^{t} \bigg(\int_0^s\bigg( \abs{c_\Delta(v)}(s-v)^{-2\alpha} + \int_0^{v} \abs{c_\Delta(v)-c_\Delta(z)}h(v,z)
(s-z)^{-\alpha}dz\bigg)dv\bigg)^2 g(t,s)\ind{s}ds \\
\le C \int_0^t\bigg(\int_0^s\bigg( \aabs{X^{N,R'}_v-X^{N,R''}_v}(s-v)^{-2\alpha}
+ \int_0^{v}\Bigg(\aabs{X^{N,R'}_{v,z}-X^{N,R''}_{v,z}}
+\aabs{X^{N,R'}_v-X^{N,R''}_v} (v-z)^\beta\\
+ \aabs{X^{N,R'}_v-X^{N,R''}_v}
\big(\aabs{X^{N,R'}_{v,z}} + \aabs{X^{N,R''}_{v,z}}\big)
\Bigg)h(v,z)(s-z)^{-\alpha}dz\ind{v} dv\bigg]^2 g(t,s)ds \\
\le C \int_0^t\bigg[\int_0^s\bigg( \aabs{X^{N,R'}_v-X^{N,R''}_v}\big((s-v)^{-2\alpha} + (s-v)^{2\beta-2\alpha} + (R'+R'')(s-v)^{-2\alpha}\big)
\\ + \int_0^{v} \aabs{X^{N,R'}_{v,z}-X^{N,R''}_{v,z}}h(v,z)dz(s-v)^{-2\alpha}\bigg)\ind{v} dv\bigg]^2g(t,s)ds\le C \int_0^t \Delta_s\, g(t,s)ds.
\end{gather*}

Summing up and taking expectations, we arrive to
\begin{equation}
\label{Ia+Ic}
\begin{gathered}
\ex{I_a' + I_a'' + I_c' + I_c''}%
\le C\int_0^t \ex{\Delta_s} g(t,s)ds.
\end{gathered}
\end{equation}

Now turn to $I_b'$ and $I_b''$.
\begin{equation}
\label{Ib}
\begin{gathered}
\ex{\cI_b(s)^2\ind{s}} =  \ex{\bigg(\int_0^{s} b_\Delta(u) dW_u\bigg)^2\ind{s} }
\le \int_0^s \ex{b_\Delta(u)^2\ind{u}}du\\
\le C\int_0^s \ex{(X^{N,R'}_u-X^{N,R''}_u)^2\ind{u}}du,
\end{gathered}
\end{equation}whence
\begin{gather*}
\ex{I_b'}\le \int_0^t \ex{\Delta_s}g(t,s) ds.
\end{gather*}
Further,
\begin{gather*}
\ex{I_b''} = \int_0^t \ex{\bigg(\int_0^{s}\abs{ \int_u^{s} b_\Delta(v) dW_v} (s-u)^{-1-\alpha}du\bigg)^2
\ind{s}}g(t,s)ds\\
\le
\int_0^t \int_0^{s}\ex{\bigg(\int_u^{s} b_\Delta(v) dW_v\bigg)^2\ind{s}} (s-u)^{-3/2-\alpha} du
\int_0^s (s-u)^{-1/2-\alpha} du \,g(t,s)ds
\\
\le C \int_0^t \int_0^s \int_u^s \ex{ \aabs{X^{N,R'}_v-X^{N,R''}_v}^2\ind{v}}dv (s-u)^{-3/2-\alpha}du\,g(t,s)ds\\
\le C\int_0^t\int_0^s  \ex{\aabs{X^{N,R'}_v-X^{N,R''}_v}^2 \ind{v}}(s-v)^{-1/2-\alpha}dv\,g(t,s)ds
\le
C\int_0^t \ex{\Delta_s}g(t,s)ds.
\end{gather*}
Combining this with \eqref{Ia+Ic}, we get
\begin{gather*}
\ex{\Delta_t}%
\le C  \int_0^t \ex{\Delta_s} g(t,s)ds,
\end{gather*}
whence we get $\Delta_s = 0$ a.s., which implies $X^{N,R'}_t=X^{N,R''}_t$ for $t<\tau_N\wedge \tau_{N,R'}\wedge \tau_{N,R'}$.

This implies, in particular, that $ \tau_{N,R''}\ge \tau_{N,R'}$ a.s. On the other hand, almost surely $\tau_{N,R} = T$ for all $R$ large enough. Indeed, assuming the contrary, for some $t\in[0,T)$ \ $P(\forall R\ge 1\ \tau_{N,R}<T)=c>0$ and $\ex{\norm{X^{R,N}}_\infty}\ge cR$,
contradicting Lemma~\ref{lem-aprior}.

It follows that there exists a process $\set{X^{N}_t,t\in[0,T]}$ such that for each $R\ge 1$ and $t\le \tau_{N,R}$ \ $X^{N}_t=X^{N,R}_t$.
Hence, it is evident that $X^N$ solves \eqref{main} with $Z$ replaced by $Z^N$.

Since $\tau_N$ increases with $N$  and eventually equals $T$, we have that there exists a process which solves initial equation \eqref{main}.

Exactly as above, one can argue that any solution to \eqref{main} is a solution to $(N,R)$-equation for $t<\tau_N\wedge \tau_{N,R}$, which gives uniqueness.
\end{proof}

\section{Limit theorem}
%
Let coefficients of \eqref{main} satisfy \eqref{assumpt}, and let  $X$ be its unique solution.
Let also $\oX$ be the solution to stochastic differential equation
\begin{equation}
\label{sde-approx}
\oX_t =X_0 +\int_0^t a(s,\oX_{s})ds+
\int_0^tb(s,\oX_{s})dW_s+\int_0^tc(s,\oX_{s})d\oZ_s,
\end{equation}
where $\oZ$ is a $\gamma$-H\"older continuous process.

As above,
for $Y\in \set{Z,\oZ}$ define a stopped
process $Y^N_t = Y_{t\wedge \tau_N}$, where  $\tau_N = \inf \set{t: \norm{Y}_{0,t} \ge N}\wedge T$, and let
$X^N$ and $\oX^N$ be the solutions to corresponding equations. Denote $A_t^{N,R} = \set{\norm{X^N}_{\infty,t} + \norm{\oX^N}_{\infty,t} \le R}$.
\begin{lemma}Under assumptions \eqref{assumpt},
$$
\ex{\norm{X^N - \oX^N}^2_{2,T}\ind{B_T^{N,R}}}\le C_{N,R}\ex{\norm{Z^N-\oZ^N}_{0,T}}
$$
with the constant $C_{N,R}$ independent of $Z$, $\oZ$.
\end{lemma}
\begin{proof}
We will use the same notation as in the proof of Theorem \ref{thm-exists-unique}, except now  $\ind{t}=\ind{A_t^{N,R}}$.

Denote $\Delta_t = \int_0^t \norm{X^N-\oX^N}^2_{s}\ind{s} g(t,s) ds$. Exactly as in the proof of Theorem \ref{thm-exists-unique}
it can be  shown that
\begin{equation}\label{deltat}
 \ex{\Delta_t}%
\le  C\left( C_{N,R}\int_0^t \ex{\Delta_s} g(t,s)ds + \ex{I_Z'} + \ex{I_Z''}\right),
\end{equation}
where \begin{gather*}
 I_Z' = \int_0^{t}\cI_Z(s)^2 g(t,s)\ind{s} ds,\ I_Z''= \int_0^t \left(\int_0^s|\cI_Z(s)-\cI_Z(u)| h(s,u) du\right)^2 g(t,s)\ind{s} ds,\\
 \cI_Z(t) =\int_0^{t} c(s,\oX_s) d(Z_s-\oZ_s).
 \end{gather*}

By \eqref{ineq}, on $A_t^{N,R}$
\begin{equation}
\label{Iz}
\begin{gathered}
\cI_Z(s)^2 \le C \norm{Z^N-\oZ^N}_{0,T}^2 \bigg(\int_0^s\bigg(\abs{c(u,\oX^N_u)} u^{-\alpha} + \int_0^u\abs{c(v,\oX^N_v)-c(u,\oX^N_u)}h(u,v)dv \bigg)du\bigg)^2\\
\le C \norm{Z^N-\oZ^N}_{0,T} \bigg(\int_0^s\bigg(\big(1+\abs{\oX_u}\big) u^{-\alpha} + \int_0^u\big((u-v)^{\beta} + \abs{\oX_u- \oX_v}\big)h(u,v)dv \bigg)du\bigg)^2\\ \le C\norm{Z^N-\oZ^N}_{0,T}^2\int_0^t \big(1+ \norm{\oX}_{\infty,s}\big)^2\ind{s}g(t,s) ds \le CR^2 \norm{Z^N-\oZ^N}_{0,T}^2.
\end{gathered}
\end{equation}
Hence,
$$
I_Z'\le CN^2 \norm{Z^N-\oZ^N}_{0,T}^2.
$$
Similarly,
\begin{gather*}
I_Z'' \le C \norm{Z^N-\oZ^N}_{0,T}^2 \int_0^t \bigg[ \int_0^s \int_u^s\bigg( \abs{c(v,\oX_v)}(v-u)^{-\alpha}\\ + \int_u^v\abs{c(v,\oX_v)-c(z,\oX_z)}h(v,z)dz \bigg)dv\,h(s,u)du  \bigg]^2 g(t,s) \ind{s} ds\\
\le C \norm{Z^N-\oZ^N}_{0,T}^2 \int_0^t \bigg[\int_0^s\int_u^s\norm{X}_{\infty,v}(v-u)^{-\alpha}dv\, h(s,u)du \bigg]^2 g(t,s) \ind{s} ds
 \le CR^2 \norm{Z^N-\oZ^N}_{0,T}^2.
\end{gather*}

Summing these estimates with \eqref{deltat}, we obtain
\begin{gather*}
\ex{\Delta_t}
\le C_{N,R} \bigg(\norm{Z^N-\oZ^N}_{0,T}^2 + \int_0^t \ex{\Delta_s} g(t,s)ds\bigg),
\end{gather*}
whence we get the statement by the generalized Gronwall lemma.
\end{proof}

The proof of the following fact uses the Burkholder inequality and the same ideas as before, so we skip it.
\begin{corollary}\label{corol-est-sup}
 For $N>1$ the estimate holds
$$
\ex{\sup_{t\in[0,T]}\abs{X - \oX}^2\ind{A_T^{N,R}}}\le C_{N}\ex{\norm{Z^N-\oZ^N}_{0,T}^2\ind{A_T^{N,R}}}
$$
with the constant $C_{N}$ independent of $Z$, $\oZ$.
\end{corollary}

Finally, we formulate a limit theorem for mixed stochastic differential equation \eqref{main}.

Let $\set{Z^n,n\ge 0}$ be a sequence of $\gamma$-H\"older continuous processes. Consider a sequence of
stochastic differential equations
\begin{equation}
\label{approx1} X^n_t =X_0 +\int_0^t a(s,X^n_{s})ds+
\int_0^tb(s,X^n_{s})dW_s+\int_0^tc(s,X^n_{s})dZ^n_s,\
t\in[0,T].
\end{equation}
\begin{theorem}\label{thm-main}
Assume that $\norm{Z-Z^n}_{0,T}\to 0$ in probability. Then $X^n_t \to X_t$ in probability uniformly in $t$.
\end{theorem}
\begin{proof}
Let $B_t^{n,N}= \set{\norm{X}_{\infty,t} + \norm{X^n}_{\infty,t} + \norm{Z}_{0,t} + \norm{Z^n}_{0,t}\le N}$, $\Delta^n = \sup_{t\in[0,T]}\abs{X^n_t-X_t}$.

For $\eps>0$ write $$
P(\Delta^n>\eps)\le P\left(\set{\Delta^n >\eps}\cap B_T^{n,N}\right)+P(\Omega \setminus B_T^{n,N}).
$$
From the assumption it is easy to see that $\ex{\norm{Z-Z^n}_{0,T}^2\ind{B_t^{N,n}}}\to 0$, $n\to \infty$.
Then by \eqref{corol-est-sup} we have for any $\eps>0$ $$P\left(\set{\Delta^n >\eps}\cap B_T^{n,N}\right)\to 0, \quad n\to\infty.$$
So
\begin{equation}\label{prob}
\limsup_{n\to\infty} P(\Delta^n>\eps)\le \limsup_{n\to\infty} P(\Omega \setminus B_T^{n,N}).
\end{equation}
We know that $\Lambda_T(Z)<\infty$ a.s., so by assumption, $\norm{Z^n}_{0,T}$ are bounded in probability uniformly in $n$. Therefore by Lemma \ref{lem-aprior}, $\norm{X^n}_{\infty,T}$ are bounded in probability uniformly in $n$ and $\norm{X}_{\infty,T}$ is finite a.s. Consequently,
$P(\Omega \setminus B_T^{n,N})\to 0$, $N\to\infty$ uniformly in $n$. Thus, we conclude the proof by sending $N\to\infty$ in \eqref{prob}.
\end{proof}
\begin{remark}
Under the assumptions of
Theorem \ref{thm-main} we have also the convergence in probability
$\norm{X-X^n}_{2,T}\to 0$, $n\to\infty$.
\end{remark}
\vskip3mm

\bibliographystyle{elsarticle-harv}
\bibliography{mbfbm-limit}
\end{document}